\documentclass[11pt]{amsart}
\usepackage{amssymb, amsthm, graphicx, epsfig, verbatim, amscd,
enumerate, stmaryrd, amsmath, psfrag}

\newtheorem{thm}{Theorem}[section]

\newtheorem{lem}[thm]{Lemma}
\newtheorem{prop}[thm]{Proposition}

\theoremstyle{definition}
\newtheorem{defn}[thm]{Definition}

\theoremstyle{remark}
\newtheorem{rem}[thm]{Remark}

\numberwithin{equation}{section}

\newcommand{\al}{\alpha}
\newcommand{\be}{\beta}
\newcommand{\ga}{\gamma}
\newcommand{\Ga}{\Gamma}

\newcommand{\de}{\delta}
\newcommand{\ep}{\varepsilon}

\newcommand{\ka}{\kappa}

\newcommand{\si}{\sigma}
\newcommand{\Si}{\Sigma}

\newcommand{\va}{\varphi}
\newcommand{\csi}{\xi}

\newcommand{\x}{\times}

\newcommand{\CC}{\mathcal C}

\renewcommand{\SS}{{\mathfrak {F}}}

\newcommand{\bb}{{\mathfrak {b}}}
\renewcommand{\i}{{\rm I}}
\newcommand{\ii}{{\rm II}}
\newcommand{\iii}{{\rm III}}
\newcommand{\iv}{{\rm IV}}
\newcommand{\imm}{{\mathrm {Imm}}}

\newcommand{\Z}{\mathbb Z}

\newcommand{\R}{\mathbb R}

\newcommand{\RP}{{\mathbb R}{P}}
\newcommand{\CP}{{\mathbb C}{P}}

\newcommand{\del}{\partial}

\newcommand{\co}{\colon\thinspace}

\newcommand{\plto}[1]
 {\xrightarrow{#1}}

\begin{document}
\mathsurround=1pt 
\title{Cobordisms of fold maps of $4$-manifolds into the space}
%Cobordism of simple fold maps}


\thanks{The author has been supported by JSPS}

\author{Boldizs\'ar Kalm\'{a}r}

\address{Kyushu University, Faculty of Mathematics, 6-10-1 Hakozaki, Higashi-ku, Fukuoka 812-8581, Japan}
\email{kalmbold@yahoo.com}

%\address{E\"{o}tv\"{o}s Lor\'{a}nd University, P\'{a}zm\'{a}ny P\'{e}ter
%s\'{e}t\'{a}ny 1/c.  H-1117 Budapest, Hungary}
%\email{kalmbold@cs.elte.hu}

\begin{abstract}
We compute the oriented cobordism group of fold maps of $4$-manifolds into $\R^3$ with all the possible restrictions (and also with no restriction) to the singular fibers. We also give geometric invariants which describe completely the cobordism group of fold maps. 
\end{abstract}

\maketitle

\section{Introduction}

Folds maps on oriented $4$-manifolds into $\R^3$ can be considered as additional structures on $4$-manifolds like framings of the stable tangent bundle of the $4$-manifold \cite{An3}, surfaces with special properties embedded into the $4$-manifold  \cite{Sa5}, and conditions about the (co)homologies of the $4$-manifold \cite{Sad, Sa5}. Like in the case of other additional structures, e.g. spin structures, one can define a corresponding notion of cobordism, i.e. cobordism of fold maps of $4$-manifolds into $\R^3$. Regarding the applications of cobordisms of $4$-manifolds equipped with additional structures, it seems to be useful to consider the cobordisms of fold maps of $4$-manifolds into $\R^3$ with prescribed singular fibers. Computing these cobordism groups is the goal of the present paper.

In \cite{Kal2}, we computed the oriented cobordism group of fold maps of $3$-manifolds into the plane, and by checking the values of the geometric cobordism invariants defined in \cite{Kal6, Kal7} on the generators, we obtain that the fold cobordism class of a fold map of a $3$-manifold into the plane is described by the cobordism classes of the immersions obtained by restricting the fold map to its definite and indefinite singular sets, respectively \cite{Kal6}. In this paper, we obtain that the fold cobordism class of a $4$-manifold into $\R^3$ is determined by the values of the geometric invariants corresponding to the singular set of the fold map, similarly to the $3$-dimensional case. Moreover, we obtain a clear and complete picture about the cobordism groups of fold maps of $4$-manifolds with prescribed singular fibers.  Some of the results of the present paper were obtained in \cite{Kal3}.
We proved \cite{Kal6} that the cobordism classes of simple fold maps of $(n+1)$-dimensional manifolds into an $n$-dimensional manifold are described by the geometric invariants defined in \cite{Kal6, Kal7}. 
Sadykov proved \cite{Sad4} that the cobordism classes of fold maps of $5$-manifolds into $\R^4$ are described by the geometric invariants defined in \cite{Kal6, Kal7}. 

We would like to thank to O. Saeki, R. Sadykov and T. Yamamoto for asking so much and
listening carefully to the author's talks on the present subject at seminars in the last 2 years.

\subsection{Notations}
In this paper the symbol ``$\amalg$'' denotes the disjoint union, 
$\ga^1$ denotes the universal line bundle over $\RP^{\infty}$,
$\ep^1_X$ (shortly $\ep^1$) denotes the trivial line bundle over the space $X$,
and the symbols $\csi^k$, $\eta^k$, etc. usually denote $k$-dimensional real vector bundles.
The symbols det$\csi^k$ and $T\csi^k$ denote the determinant line bundle 
and the Thom space of the bundle $\csi^k$, respectively.
The symbol $\imm^{\csi^k}_{N}(n-k,k)$ denotes
the cobordism group of $k$-codimensional immersions into 
an $n$-dimensional manifold $N$
whose normal bundles are induced from $\csi^k$ (this
group is isomorphic to the group $\{\dot N, T\csi^k \}$, where
$\dot N$ denotes the one point compactification of the manifold $N$
and the symbol $\{X,Y \}$ denotes the group of stable homotopy classes of continuous
maps from the space $X$ to the space $Y$).
The symbol $\imm^{\csi^k}(n-k,k)$ denotes
the cobordism group of $k$-codimensional immersions into $\R^n$ 
whose normal bundles are induced from $\csi^k$ (this
group is isomorphic to $\pi_{n}^s(T\csi^k)$).
The symbol $\imm_N(n-k,k)$ denotes
the cobordism group $\imm^{\ga^k}_N(n-k,k)$ where
$\ga^k$ is the universal bundle for $k$-dimensional real vector bundles and 
$N$ is an $n$-dimensional manifold.
The symbol $\pi_n^s(X)$ ($\pi_n^s$) denotes the $n$th stable homotopy group of the space $X$ (resp. spheres).
The symbol ``id$_A$'' denotes the identity map of the space $A$.
The symbol $\ep$ denotes a small positive number.
All manifolds and maps are smooth of class $C^{\infty}$.

\section{Preliminaries}

\subsection{Fold maps}

Let $Q^{n+1}$ and $N^n$ be smooth manifolds of dimensions $n+1$ and $n$ 
respectively. Let $p \in Q^{n+1}$ be a singular point of 
a smooth map $f \co Q^{n+1} \to N^{n}$. The smooth map $f$  has a {\it fold 
singularity} at the singular point $p$ if we can write $f$ in some local coordinates around $p$  
and $f(p)$ in the form 
\[  
f(x_1,\ldots,x_{n+1})=(x_1,\ldots,x_{n-1}, x_n^2 \pm x_{n+1}^2).
\] 
A smooth map $f \co Q^{n+1} \to N^{n}$ is called a {\it fold map} if $f$ has only 
fold singularities.

A smooth map $f \co Q^{n+1} \to N^n$ 
  has a {\it definite fold
singularity} at a fold singularity $p \in Q^{n+1}$ if we can write $f$ in some local coordinates around $p$  
and $f(p)$ in the form 
\[  
f(x_1,\ldots,x_{n+1})=(x_1,\ldots,x_{n-1},x_n^2 + x_{n+1}^2),
\] 
otherwise $f$ has an {\it indefinite fold singularity} at the fold singularity $p \in Q^{n+1}$.
Let $S_1(f)$ denote the set of indefinite fold singularities of $f$ in $Q^{n+1}$ and
 $S_0(f)$ denote the set of definite fold singularities of $f$ in $Q^{n+1}$.
 Let $S_f$ denote the set $S_0(f) \cup S_1(f)$.
Note that the set $S_f$ is an ${(n-1)}$-dimensional submanifold of the manifold
$Q^{n+1}$.
 Usually without mentioning we suppose that the source manifold $Q^{n+1}$ 
and the target manifold $N^{n}$ are 
oriented.
 
%For fixed integers $q$ and $n$, $q \ge n \ge 0$, let $FOLD^{}(q, N^n)$ 
%denote the set of the fold maps $f \co Q^q \to N^n$, where $Q^q$
%runs over the set of all closed {\it oriented} $q$-dimensional manifolds.
%Let $f \co Q^q \to N^n$ be a fold map.
%Then the integer $n-q$ is called the {\it codimension} of the map $f$. 

If $f \co Q^{n+1} \to N^n$ is a fold map in general position, then 
the map $f$
restricted to the singular set $S_f$ is a general positional
 codimension one immersion  into the target manifold $N^n$.

Since every fold map is in general position after a small perturbation, 
and we study maps under the equivalence relations {\it cobordism} 
%and {\it bordism} 
(see Definitions~\ref{cobdef}),
in this paper we can restrict ourselves to studying fold maps which are 
in general position.
Without mentioning we suppose that a fold map $f$ is in general position.

\begin{defn}
For an integer $k > 0$ a 
fold map $f \co Q^{n+1} \to N^{n}$ is called a {\it $k$-simple fold map}
if every connected 
component of an arbitrary fiber of the map $f$ contains at most $k$ singular points. 
(For $k=1$ we say shortly {\it simple fold map}.)
\end{defn}

%\subsection{Stein factorization}
% 
%We use the notion of the Stein factorization of a smooth map 
%$f\co Q^q \to N^n$, where $Q^q$ and $N^n$ are smooth manifolds 
%of dimensions $q$ and $n$ respectively $(q \geq n)$. 
%Two points  $p_1,p_2 \in Q^q$ are {\it equivalent} if  
%$p_1$ and $p_2$ lie on the same component of an $f$-fiber.
%Let $W_f$ denote the quotient space of $Q^q$ with respect
%to this equivalence relation and $q_f \co Q^q \to W_f$ the quotient map.
%Then there exists a unique continuous map $\Bar{f} \co W_f \to N^n$ such that
%the diagram
%\begin{center}
%\begin{graph}(6,2)
%\graphlinecolour{1}\grapharrowtype{2}
%\textnode {A}(1,1.5){$Q^q$}
%\textnode {B}(5, 1.5){$N^n$}
%\textnode {C}(3, 0){$W_f$}
%\diredge {A}{B}[\graphlinecolour{0}]
%\diredge {C}{B}[\graphlinecolour{0}]
%\diredge {A}{C}[\graphlinecolour{0}]
%\freetext (3,1.8){$f$}
%\freetext (1.4, 0.6){$q_f$}
%\freetext (4.4,0.6){$\Bar f$}
%\end{graph}
%\end{center}
%commutes or in other words
%$f = \Bar{f} \circ q_f$. The space $W_f$ or the factorization of the 
%map $f$ into the composition of $q_f$ and $\Bar{f}$ is called the {\it Stein
%factorization} of the map $f$. We call the map $\Bar{f}$ the  
%{\it Stein factorization} of the map $f$ as well. Note that if $f$ is a 
%generic smooth map of a closed 
%$q$-dimensional manifold into $N^n$ (for example, if $f$ is a fold map in general position),
% then its Stein factorization $W_f$ is a  
%compact $n$-dimensional CW complex.

\subsection{Equivalence relations of fold maps}\label{kob}

\begin{defn}\label{cobdef} (Cobordism) 
Two fold maps $f_i \co Q_i^{n+1} \to N^n$ $(i=0,1)$  
of closed (oriented) $({n+1})$-dimensional manifolds $Q_i^{n+1}$ $(i=0,1)$ 
into an $n$-dimensional manifold $N^n$ are  
{\it (oriented) cobordant} if 
\begin{enumerate}[a)]
\item
there exists a fold map 
$F \co X^{n+2} \to N^n \times [0,1]$ of a compact (oriented) $(n+2)$-dimensional 
manifold $X^{n+21}$, %with 
\item
$\del X^{n+2} = Q_0^{n+1} \amalg (-)Q_1^{n+1}$ and %such that 
\item
${F \mid}_{Q_0^{n+1} \x [0,\ep)}=f_0 \x
{\mathrm {id}}_{[0,\ep)}$ and ${F \mid}_{Q_1^{n+1} \x (1-\ep,1]}=f_1 \x 
{\mathrm {id}}_{(1-\ep,1]}$, where 
$Q_0^{n+1} \x [0,\ep)$
 and $Q_1^{n+1} \x (1-\ep,1]$ are small collar neighbourhoods of $\del X^{n+2}$ with the
identifications $Q_0^{n+1} = Q_0^{n+1} \x \{0\}$ and $Q_1^{n+1} = Q_1^{n+1} \x \{1\}$. 
%where $\ep$  refers to a small positive number. 
\end{enumerate}

We call the map $F$ a {\it cobordism} between $f_0$ and $f_1$.
\end{defn} 
This clearly defines an equivalence relation on the set of fold maps 
of closed (oriented) $({n+1})$-dimensional manifolds into an  
$n$-dimensional manifold $N^n$.

We denote 
 the set of fold (oriented) cobordism classes of 
fold maps of closed (oriented) $({n+1})$-dimensional manifolds 
into an $n$-dimensional manifold 
$N^n$ (into the Euclidean space $\R^n$)
by $\CC ob_{f}^{(O)}(N^n)$ (by $\CC ob_{f}^{(O)}(n)$).
We note that we can define a commutative semigroup operation in the usual way on the 
set of cobordism classes $\CC ob_{f}^{(O)}(N^n)$
by the disjoint union.
In the case of $N^n = \R^n$ this semigroup operation is equal to the usual group operation,
i.e., the far away disjoint union.

We can refine this equivalence relation by considering the 
singular fibers (see, for example, \cite{Lev, Sa, SaYa, Ya}) of a fold map.

\begin{defn}
Let $\tau$ be a set of singular fibers.
Two fold maps $f_i \co Q_i^{n+1} \to N^n$ $(i=0,1)$ with singular fibers in the set $\tau$  
of closed (oriented) $({n+1})$-dimensional manifolds $Q_i^{n+1}$ $(i=0,1)$ 
into an $n$-dimensional manifold $N^n$ are  
{\it (oriented) $\tau$-cobordant} if 
they are (oriented) cobordant in the sense of Definition~\ref{cobdef}
by a fold map $F \co X^{n+2} \to N^n \times [0,1]$ 
whose singular fibers are in the set $\tau$.
\end{defn}

We denote the set of $\tau$-cobordism classes of fold maps with singular fibers in the set $\tau$  
by $\CC ob_{\tau}^{(O)}(N^n)$.

In this way for an integer $k>0$ we can obtain the notion of {\it $k$-simple fold cobordism} of 
{\it $k$-simple fold maps}, i.e., let
$\tau_k$ be the set of all the singular fibers which have at most $k$ singular points in each of the connected 
components of their fibers. We denote the set of $k$-simple fold (oriented) cobordism classes of $k$-simple fold maps 
of closed (oriented) $({n+1})$-dimensional manifolds $Q^{n+1}$  
into an $n$-dimensional manifold $N^n$
by $\CC ob_{s(k)}^{(O)}(N^n)$. For results about simple fold maps, see, for example, 
\cite{Kal6, Sa1, Sa4, Saku, Yo}.

\subsection{Cobordism invariants of fold maps}\label{cobinv}

In \cite{Kal6, Kal7} for a 
fold map $f \co Q^{n+1} \to N^{n}$
we defined
homomorphisms
\[
\csi_{{\mathrm {indef}}, n}^N  \co \CC ob_{f}^O(N^n) \to \imm^{{\mathrm {det}}(\ga^1 \x \ga^1)}_N(n-1,1)
\]
and
\[
\csi_{{\mathrm {def}}, n}^N  \co \CC ob_{f}^O(N^n) \to \imm^{\ep^1_{\CP^{\infty}}}_N(n-1,1)
\]
by mapping a cobordism class of a fold map $f$ into the cobordism classes of the immersions of its 
indefinite and definite singular set $S_1(f)$ and $S_0(f)$, respectively, with normal bundles induced from
the bundle $\eta_{\mathrm {indef}}^1 \co $det$(\ga^1 \x \ga^1) \to \RP^{\infty} \x \RP^{\infty}$
 and the bundle $\eta_{\mathrm {def}}^1 \co \ep^1 \to \CP^{\infty}$, respectively, where
the bundle $\eta_{\mathrm {indef}}^1 \co $det$(\ga^1 \x \ga^1) \to \RP^{\infty} \x \RP^{\infty}$ 
and the bundle 
$\eta_{\mathrm {def}}^1 \co \ep^1 \to \CP^{\infty}$
are the targets of the universal indefinite and definite germ bundles (see \cite{Kal6, Kal7}).
By \cite{Kal6} we have a homomorphism
\[
\theta_n^N \circ \csi_{{\mathrm {indef}}, n}^N \co
\CC ob_{f}^O(N^n) \to 
 \imm_N(n-1,1) \oplus \imm^{\ga^1 \x \ga^1}_N(n-2,2)
\]
which maps a fold cobordism class $[f]$ into the sum of the cobordism class of the immersion of the indefinite 
singular set $S_1(f)$ and the cobordism class of the 
``twisting'' of the indefinite germ bundle over it.

In this paper let us denote shortly the homomorphisms $\theta_n^N \circ \csi_{{\mathrm {indef}}, n}^N$ 
and $\csi_{{\mathrm {def}}, n}^N$ by 
$\iota$ and $\de$,
%$\zeta_{{\mathrm {indef}}}$ and $\csi_{{\mathrm {def}}}$
respectively. Summarizing, we have a homomorphism
\[
\de \oplus \iota \co \CC ob_{f}^O(N^n)
\to \imm^{\ep^1_{\CP^{\infty}}}_N(n-1,1) \oplus  \imm_N(n-1,1) \oplus \imm^{\ga^1 \x \ga^1}_N(n-2,2)
\]
which can be interpreted as a homomorphism $\de^{imm} 
\oplus \de^{tw} 
\oplus \iota^{imm} \oplus \iota^{tw}$ into
\[
\imm^{\ep^1}_N(n-1,1) \oplus \{ \dot N, S\CP^{\infty} \} \oplus \imm_N(n-1,1) \oplus \imm^{\ga^1 \x \ga^1}_N(n-2,2)
\] 
where the homomorphisms $\de^{imm}$ and
$\iota^{imm}$ map a fold cobordism class $[f]$
into the cobordism class of the immersion of the definite and indefinite singular set of $f$, respectively,
and the homomorphisms $\de^{tw}$ and
$\iota^{tw}$ map 
a fold cobordism class $[f]$
into the cobordism class of the ``twisting'' of the germ bundle 
over the immersion of the definite and indefinite singular set, respectively.

\subsection{Fiber-singularities, Bundle structure and punctured fold maps}\label{apparatus}

Throughout the paper, we use the notions and statements of \cite{Kal2, Kal5, Kal6} about
(punctured) (multi)fiber-singularities \cite{Kal6}, bundle structures of fold maps \cite{Kal6}, puncturing fold maps \cite{Kal2}, and Pontryagin-Thom type construction for fold maps \cite{Kal5} and for Stein facorizations \cite{Kal2}.
By these tools it is enough to deal with cobordisms of punctured Stein factorizations with the appropriate symmetry groups (symmetries of Stein factorizations of fiber-singularities, which come from symmetries of fiber-singularities).

\subsection{Symmetry groups of the fiber-singularities} 

Recall \cite{Kal6} that the symmetry group $ISO^O(\si_{\SS})$ of a punctured indefinite fiber-singularity $\si_{\SS}$
can be reduced to a finite group which can be determined by the symmetries of the fiber
$\si_{\SS}^{-1}(0)$.

\begin{prop}
The symmetry groups of the punctured indefinite fiber-singularities are the following.
$ISO^O(\si_{\i^1}) = \Z_2$, 
$ISO^O(\si_{\ii^2}) = \Z_2$,
$ISO^O(\si_{\ii^3}) = D_4$,
$ISO^O(\si_{\iii^4}) = \Z_2$,
$ISO^O(\si_{\iii^5}) = \Z_3$,
$ISO^O(\si_{\iii^6}) = D_3$,
$ISO^O(\si_{\iii^7}) = \{0\}$,
$ISO^O(\si_{\iii^8}) = \Z_3 \oplus \Z_2^3$.
\end{prop}

\section{Main results}

%\subsection{Cobordisms of fold maps on $4$-manifolds}

By Section~\ref{cobinv}, we have a homomorphism
\begin{multline*}
\de \oplus \iota = 
\de^{imm} 
\oplus \de^{tw} 
\oplus \iota^{imm} \oplus \iota^{tw}
\co  \CC ob_{f}^{O}(4,-1)
\to \\
\imm^{\ep^1}(2,1) \oplus \pi^s_2( \CP^{\infty} ) \oplus \imm(2,1) \oplus \imm^{\ga^1 \x \ga^1}(1,2),
\end{multline*}
i.e., a homomorphism $$\de^{imm}
\oplus \de^{tw} \oplus \iota^{imm} \oplus \iota^{tw}
\co  \CC ob_{f}^{O}(4,-1)
\to \Z_2 \oplus \Z \oplus \Z_8 \oplus \Z_4$$
(for the group $\imm^{\ga^1 \x \ga^1}(1,2)$, see Lemma~\ref{immgrouplem} below).

Let $c \co \imm^{\ep^1}(2,1) \to \imm(2,1)$ ($c \co \Z_2 \to \Z_8$) denote the natural inclusion homomorphism.
%For $a \in \Z$ %$b \in \Z_8$ 
%and $l \in \N$ let $a_l$% and $b_4$
%denote $a \bmod l$.% and $b$ $({\mathrm {mod}} \  4)$ respectively.

\begin{thm}\label{cob4to3}
The homomorphism $\de^{imm} \oplus \de^{tw} \oplus \iota^{imm} \oplus \iota^{tw}$
gives a complete invariant of the cobordism goup $\CC ob_{f}^{O}(4,-1)$,
%The fold cobordism group $\CC ob_{f}^{O}(4,-1)$ 
which is isomorphic to $\Z_2 \oplus \Z_4 \oplus \Z$.
The isomorphism can be given by the homomorphism
\[
\de^{imm} %\oplus
%\frac{(c \circ \de^{imm} + \iota^{imm})_4 - 
%\iota^{tw}}{2}
\oplus \frac{c \circ \de^{imm} + \iota^{imm} - 
(\de^{tw}/3 \bmod 8)}{2}
\oplus (\de^{tw}/3),
\]
which coincides with the homomorphism defined by
\[
[f \co M^4 \to \R^3] \mapsto 
\de^{imm}([f]) %\oplus
%\frac{[f|_{S_f}]_4 - 
%\iota^{tw}([f])}{2}
\oplus \frac{[f|_{S_f}] - (\si(M^4) \bmod 8)}{2}
\oplus \si(M^4),
\]
where 
$\si(M^4)$ denotes the signature of the oriented source manifold $M^4$
and $[f|_{S_f}] \in \Z_8$ is the cobordism class of the 
immersion of the singular set of the fold map $f$ into $\R^3$.
\end{thm}
%\begin{cor}
%The homomorphisms $\de^{imm}$, $\de^{tw}$, $\iota^{imm}$ and $\iota^{tw}$
%give complete invariants of the cobordism goup $\CC ob_{f}^{O}(4,-1)$.
%\end{cor}

\begin{rem}
In \cite{Kal2, Kal6}, we showed that the homomorphism 
$\de^{imm} \oplus \iota^{imm}$ gives an isomorphism 
between the fold cobordism goup $\CC ob_{f}^{O}(3,-1)$ and $\Z_2^2$, where
instead of $\de^{imm} \oplus \iota^{imm}$, we can also choose 
the homomorphism $\iota^{imm} \oplus \iota^{tw}$ as an isomorphism, as one can see easily 
by checking the values on the generators.
\end{rem}

%In Theorem~\ref{cob4to3} by [...] the value $\frac{1}{3}\de^{tw}([f])$ is equal to the 
%signature $\si(M^4)$ of the source manifold $M^4$ of a fold map $f \co M^4 \to \R^3$.
%Moreover, the value $(c \circ \de^{imm} + \iota^{imm})([f]) \in \Z_8$
%is equal to the cobordism class $[f|_{S_f}]$ of the
%immersion of the singular set of the fold map $f$ into $\R^3$.
The statement of Theorem~\ref{cob4to3}
includes implicitly the following.

\begin{prop}
For a fold map $f \co M^4 \to \R^3$ of a closed oriented $4$-manifold $M^4$, we have
\[
[f|_{S_f}] \equiv \si(M^4) %\equiv \iota^{tw}([f]) 
\equiv \de^{tw}([f]) \ \ ({\mathrm {mod}} \ 2), 
\]
and
$\de^{tw}([f])\equiv 0 \ \ ({\mathrm {mod}} \ 3)$.
\end{prop}

We remark that the first congruence of the above proposition can be deduced by using results of \cite{Ba, Fu} or \cite{Ya2} very easily
and the congruence $\si(M^4) \equiv \de^{tw}([f]) \ \ ({\mathrm {mod}} \ 2)$
follows from \cite{Sa1, To} as well. The congruence 
$\de^{tw}([f])\equiv 0 \ \ ({\mathrm {mod}} \ 3)$ is related to \cite{Sad, Sa5}.

\section{Computing cobordism groups of fold maps on $4$-manifolds}

\begin{lem}\label{immgrouplem}
The cobordism group
$\imm^{\ga^1 \x \ga^1}(1,2)$ is isomorphic to $\Z_4$ and
the cobordism group $\imm^{\eta_{\ii^3}}(1,2)$ is isomorphic to $\Z_4$.
\end{lem}
\begin{proof}
A representative of the generator of the group  $\imm^{\ga^1 \x \ga^1}(1,2)$
is a trivially embedded circle into $\R^3$ whose normal bundle is twisted by $180$ degrees as we go once around
the circle.

The group $\imm^{\eta_{\ii^3}}(1,2)$ has an epimorhism onto $\Z_4$ by forgetting the $s_{\ii^3}$ bundle 
of the representatives, and this epimorhism is also injective.
A generator of the  group $\imm^{\eta_{\ii^3}}(1,2)$  can be represented by an embedded circle with a $180$ degree twist in
its normal bundle.
Details are left to the reader.
\end{proof}

Now let us prove Theorem~\ref{cob4to3}.
\begin{proof}[Proof of Theorem~\ref{cob4to3}]
By \cite{Kal3} the $2$-simple fold cobordism group 
$\CC ob_{s(2)}^{O}(4,-1)$ is isomorphic to
$$\CC ob_{s}^{O}(4,-1) \oplus \imm^{\ep^1 \x \ga^1}(1,2) \oplus \imm^{\eta_{\ii^3}}(1,2),$$
which is isomorphic to $\Z_2^2 \oplus \Z_2 \oplus \Z_4$.

Let us give representatives of generators of this group and compute the values
of the homomorphism 
\[
\de^{imm} \oplus \de^{tw} \oplus \iota^{imm} \oplus \iota^{tw}
\co \CC ob_{f}^{O}(4,-1)
\to
\Z_2 \oplus \Z \oplus \Z_8 \oplus \Z_4
\]
on these representatives.

\subsection{Generators of the simple fold cobordism group $\CC ob_{s}^{O}(4,-1)$}  
The group $$\CC ob_{s}^{O}(4,-1) = \imm^{\ep^1}(2,1) \oplus 
\imm^{\ep^1 \x \ga^1}(1,2)$$ (see \cite{Kal6}) is isomorphic to $\Z_2 \oplus \Z_2$ and
the representatives of the two generators $(1,0)$ and $(0,1)$ are given 
by the punctured simple fold maps $f_i \co M_i^4 \to \R^3$ ($i=1,2$)
which are constructed as follows.

Let $g \co T^2 \to \R^3$ be the immersion of the torus with one unknotted circle
as the set of double points in $\R^3$, which represents
the non-trivial element in $\imm^{\ep^1}(2,1)$: i.e.,
 the image $g(T^2)$ is contained in a small tubular neighbourhood 
of the circle of double points and the intersection of $g(T^2)$ with
a normal $2$-disk fiber is the standard ``figure eight'', which is
rotated by $360$ degrees as it goes once around the circle of double points.
Let $M_1^4$ be the total space of a trivial $s_{\i^1}$ bundle over the torus $T^2$ and let 
$f_1$ be the fold map which maps the subbundle corresponding to
the double points of the ``figure eights'' in the fibers $s_{\i^1}$ into $\R^3$
as the immersion $g$, and maps a fiber $s_{\i^1}$ of this bundle
into a fiber of the trivial normal bundle of $g(T^2)$ as the fiber-singularity
$\si_{\i^1} \co s_{\i^1} \to J$. The construction of the fold map $f_2 \co M_2^4 \to \R^3$
is similar, but we choose the immersion $g \co T^2 \to \R^3$ to be
the standard embedding, and $M_2^4$ is the total space of a 
non-trivial $s_{\i^1}$ bundle over the torus $T^2$ such that the cobordism class
$\iota^{tw}([f])$ is equal to the element of order two in the cobordism group
$\imm^{\ga^1 \x \ga^1}(1,2)=\Z_4$.
It is easy to check that the values of the homomorphism 
$$\de^{imm} \oplus \de^{tw} \oplus \iota^{imm} \oplus \iota^{tw}
\co \CC ob_{f}^{O}(4,-1)
\to
\Z_2 \oplus \Z \oplus \Z_8 \oplus \Z_4
$$ on the fold cobordism classes $[f_1]$ and $[f_2]$ are equal to
$(1,0,4,0)$ and $(1,0,0,2)$, respectively.

\subsection{Generators of the cobordism group $\CC ob_{\i^1 + \ii^2}^{O}(4,-1)$}
By \cite{Kal3} this group is isomorphic to 
$$\CC ob_{s}^{O}(4,-1) \oplus \imm^{\ep^1 \x \ga^1}(1,2) = \Z_2^2 \oplus \Z_2.$$
Three representatives of its generators are the two representatives $f_1$ and $f_2$
of the generators of the simple fold cobordism group and a third fold map $f_3 \co M_3^4 \to \R^3$
which is constructed as follows.
Let $f'_3 \co M'^4_3 \to \R^3$ be the composition $h \circ p$ of the total space 
$p \co s_{\ii^2}^{\#} \x S^1 \to J^2 \x S^1$ of 
a trivial $\si_{\ii^2} \co s_{\ii^2}^{\#} \to J^2$ bundle over the circle $S^1$ and 
the embedding $h \co J^2 \x S^1 \to \R^3$, where $h(J^2 \x S^1)$ is the regular 
neighbourhood  
of the standard circle $h( \{ 0 \} \x S^1)$ in $\R^3$ twisted by $360$ degrees.  
The modification of the punctured Stein factorization of the fiber-singularity 
$\si_{\ii^2}$ (see \cite{Kal2}) gives us a way to extend the manifold 
$s_{\ii^2}^{\#} \x S^1$ fiberwise to a closed manifold $M_3^4$ and the map
$f'_3$ to a fold map $f_3 \co M_3^4 \to \R^3$
which has only $\si_{\i^1}$ and $\si_{\ii^2}$ as indefinite fiber-singularities
and whose indefinite singular set is a torus immersed into $\R^3$ in the same way as
the immersion $g$ in the construction of the fold map $f_1$.
It is an easy exercise to show that the
 value of the homomorphism 
$\de^{imm} \oplus \de^{tw} \oplus \iota^{imm} \oplus \iota^{tw}
\co \CC ob_{f}^{O}(4,-1)
\to
\Z_2 \oplus \Z \oplus \Z_8 \oplus \Z_4
$ on the cobordism class $[f_3]$ is equal to 
$(1,0,4,0)$. 

\subsection{Generators of the 2-simple fold cobordism group}% $\CC ob_{s(2)}^{O}(4, -1)$}  

The 2-simple fold cobordism group $\CC ob_{s(2)}^{O}(4,-1)$
is isomorphic to $$\CC ob_{\i^1 + \ii^2}^{O}(4,-1) \oplus \imm^{\eta_{\ii^3}}(1,2) = 
\Z_2^2 \oplus \Z_2 \oplus \Z_4.$$

The additional generator can be constructed as follows.
Similarly to the above let $p \co  s_{\ii^3}^{\#} \x_{\Z_2} S^1 \to J^2 \x_{\Z_2} S^1$ be
the total space of 
a $\si_{\ii^3} \co s_{\ii^3}^{\#} \to J^2$ bundle over the circle $S^1$ with
an automorphism of the fiber-singularity $\si_{\ii^3}$ which acts
as a $180$ degree rotation on $J^2$. Let $h \co J^2 \x_{\Z_2} S^1 \to \R^3$ be the
standard embedding
(as the tubular neighbourhood of $\{ 0 \} \x_{\Z_2} S^1$) into $\R^3$.
Now let $f_4 \co M_4^4 \to \R^3$ be the fold map which we obtain
by closing the manifold $s_{\ii^3}^{\#} \x_{\Z_2} S^1$ and the map
$h \circ p$ using the modification of the punctured 
Stein factorization of the fiber-singularity 
$\si_{\ii^3}$ (see \cite{Kal2}) 
similarly to the construction of the previous 
fold map $f_3$. The homomorphism $\de^{imm} \oplus \de^{tw} \oplus \iota^{imm} \oplus \iota^{tw}$
takes the value $(0,0,2,x)$ for some $x \in \Z_4$ 
on the class $[f_4]$ in $\Z_2 \oplus \Z \oplus \Z_8 \oplus \Z_4$.

%For the other generator, we chose 
%$p \co  s_{\ii^3}^{\#} \x_{\Z_2} S^1 \to J^2 \x S^1$ to be
%the total space of 
%a $\si_{\ii^3} \co s_{\ii^3}^{\#} \to J^2$ bundle over the circle $S^1$ with
%the automorphism of the fiber-singularity $\si_{\ii^3}$ which acts
%as the identity on $J^2$ and acts not identically on $s_{\ii^3}^{\#}$. 
%We chose $h \co J^2 \x S^1 \to \R^3$ to be the
%standard embedding
%(as the tubular neighbourhood of $\{ 0 \} \x S^1$) into $\R^3$.
%Now similarly to the previous ones, we obtain a fold map $f_5 \co M_5^4 \to \R^3$
%whose fold cobordism class 
%goes to $(0,0,0,2)$ under the homomorphism
%$\de^{imm} \oplus \de^{tw} \oplus \iota^{imm} \oplus \iota^{tw}$.

Summerizing, we have four generators $[f_1], \ldots, [f_4]$ of the 2-simple fold
cobordism group  $\CC ob_{s(2)}^{O}(4,-1) = 
\Z_2^3 \oplus \Z_4$ on which 
the homomorphism 
$\de^{imm} \oplus \de^{tw} \oplus \iota^{imm} \oplus \iota^{tw}$
takes the values 
\[
\begin{pmatrix}
%0 & I_{\la} \\ I_{\la} & 0 
1 & 1 & 1 & 0 \\%& 0 \\
0 & 0 & 0 & 0 \\%& 0 \\ 
4 & 0 & 4 & 2 \\%& 0 \\
0 & 2 & 0 & x \\%& 2
\end{pmatrix},
\] 
where the column vectors correspond to the basis $[f_1], \ldots, [f_4]$.

\subsection{Computing the cobordism group $\CC ob_{s(2)+\iii^4+\iii^6}^{O}(4,-1)$}  

By the above matrix it is easy to see that the kernel of the homomorphism
\begin{multline*}
\Z_2 \oplus \Z_2 \oplus \Z_2 \oplus \Z_4 = \CC ob_{s(2)}^{O}(4,-1) \plto{\va_3}  \CC ob_{s(2)+\iii^4+\iii^6}^{O}(4,-1) \plto{\psi}  \\
\CC ob_{f}^{O}(4,-1) \plto{\de^{imm} \oplus \de^{tw} \oplus \iota^{imm} \oplus \iota^{tw}}  \Z_2 \oplus \Z \oplus \Z_8 \oplus \Z_4,
\end{multline*}
where the homomorphisms $\va_3$ and $\psi$ are the natural
homomorphisms,
coincides with
\begin{enumerate}[(a)]
\item
the group of order four generated by 
$(1,0,1,0)$ and $(1,1,0,2)$ if $x$ is a generator of $\Z_4$,
\item
the group of order two  generated by 
$(1,0,1,0)$  if $x$ is not a generator of $\Z_4$.
\end{enumerate}
Now, we show the case (a) holds.
The boundaries $\dot \del \si_{\iii^4}$ and $\dot \del \si_{\iii^6}$of the punctured fiber-singularities
$\si_{\iii^4}$ and $\si_{\iii^6}$, respectively, are $2$-simple fold cobordant 
to two punctured fold maps $g_1 \co N_1^3 \to S^2$ and $g_2 \co N_2^3 \to S^2$ 
whose indefinite singular set
is immersed into $S^2$ as embedded circles and ``small figure eights'' where
the double points of the ``small figure eights'' can be the image
of a singular fiber $\ii^2$ and $\ii^{1,1}$ in the case of the fiber-singularity
$\si_{\iii^4}$, and $\ii^2$ and $\ii^3$ 
in the case of the fiber-singularity
$\si_{\iii^6}$.
Therefore, we have two punctured $2$-simple fold maps 
$g_1 \co N_1^3 \to \R^2$ and $g_2 \co N_2^3 \to \R^2$
obtained from the boundaries of the punctured fiber-singularities 
$\si_{\iii^4}$ and $\si_{\iii^6}$ respectively, 
and we have two 
$(s(2)+\iii^4+\iii^6)$-null-cobordisms $G_1$ and $G_2$ of these fold maps $g_1$ and $g_2$,
respectively,
obtained from the punctured fiber-singularities 
$\si_{\iii^4}$ and $\si_{\iii^6}$, respectively.
Now let us take the punctured fold maps
$g_i \x {\mathrm {id}}_{S^1} \co N_i^3 \x S^1 \to \R^2 \x S^1$ ($i=1,2$), where
we consider $\R^2 \x S^1$ as the embedded normal bundle in $\R^3$
of a standard embedded circle $S^1$, whose fiber is twisted by $360$ degrees 
as we go once around the embedded circle $S^1$. 
Let us consider the same construction with the null-cobordisms $G_1$
and $G_2$. In this way, we obtain two 
$(s(2)+\iii^4+\iii^6)$-null-cobordant punctured $2$-simple fold maps 
which are not $2$-simple null-cobordant and are not $2$-simple cobordant to each other
(since only one of them contains not-zero $[f_4]$ component). Hence the case (a) holds, i.e.,
the kernel of $(\de^{imm} \oplus \de^{tw} \oplus \iota^{imm} \oplus \iota^{tw}) \circ \psi \circ \va_3$ is generated by the elements
$[f_1]+[f_3]$ and $[f_1]+[f_2]+2[f_4]$. Moreover the kernel of 
$\va_3$ is also generated by these elements.

In the following, we show that the homomorphism $\va_3$ is surjective (this
implies together with the above argument that $({\de^{imm} \oplus \de^{tw} \oplus \iota^{imm} \oplus \iota^{tw}}) 
\circ \psi$ is injective and hence so is $\psi$).
By the Pontryagin-Thom-Sz\H{u}cs type construction for fold maps \cite{Kal3, Kal5}, 
we have the exact sequence\footnote{The homotopy exact sequence for the pair
$(\Ga_{s(2)+\iii^4+\iii^6},\Ga_{s(2)})$, where $\Ga_{s(2)+\iii^4+\iii^6}$ and
$\Ga_{s(2)}$
 are the Pontryagin-Thom-Sz\H{u}cs type classifying spaces for
punctured oriented ${(s(2)+\iii^4+\iii^6)}$-maps and $2$-simple fold maps, respectively.}
\begin{multline*}
\CC ob_{s(2)}^{O}(4,-1) \plto{\va_3}  \CC ob_{s(2)+\iii^4+\iii^6}^{O}(4,-1)
\plto{\al}  \pi_3(\Ga_{s(2)+\iii^4+\iii^6},\Ga_{s(2)}) \plto{\be} \\ \CC ob_{s(2)}^{O}(3,-1)
 \plto{\va_2}  \CC ob_{s(2)+\iii^4+\iii^6}^{O}(3,-1),
\end{multline*}
where the last two cobordism groups are isomorphic to $\Z_2^4$ and $\Z_2^2$ 
respectively (see \cite{Kal2}), $\va_2$ is surjective, and since we can show that the group 
$\pi_3(\Ga_{s(2)+\iii^4+\iii^6},\Ga_{s(2)})$ is isomorphic to $\Z_2^2$ by an argument similar to \cite[after Lemma~5.6]{Kal2},
we obtain that $\al$ is the null-homomorphism, and hence $\va_3$ is
surjective.

%In order to show that the kernel of $\va_3$ contains 
%the elements $(1,0,1,0,0)$ and $(1,1,0,2,x+1)$,
%we construct $(s(2)+\iii^4+\iii^6)$-null-cobordisms 
%of the 2-simple fold cobordism classes 
%$[f_1] + [f_3]$ and $[f_1]+[f_2]+2[f_4]+(x+1)[f_5]$ as follows.

%Therefore the kernel of the surjective homomorphism
%$\va_3$ is equal to the group generated by the elements
%$(1,0,1,0)$ and $(1,1,0,2)$. 
Hence the cobordism group 
$\CC ob_{s(2)+\iii^4+\iii^6}^{O}(4,-1)$
is isomorphic to $\Z_2 \oplus \Z_4$,
and a system of representatives of its generators is given by
$f_1$ and $f_4$. Note that the homomorphism
$(\de^{imm} \oplus \de^{tw} \oplus \iota^{imm} \oplus \iota^{tw}) \circ \psi \co
\CC ob_{s(2)+\iii^4+\iii^6}^{O}(4,-1) \to \Z_2 \oplus \Z \oplus \Z_8 \oplus \Z_4$ is injective. 

\subsection{Computing the $3$-simple fold cobordism group}% $\CC ob_{s(3)}^{O}(4,-1)$}  

We obtain that the cobordism group
$\CC ob_{s(2)+\iii^4+\iii^5+\iii^6+\iii^7}^{O}(4,-1)$  
is isomorphic to $\Z_2 \oplus \Z_4 \oplus \imm^{\eta_{\iii^5}}(0,3) \oplus \imm^{\ep^3}(0,3)$,
which is isomorphic to $\Z_2 \oplus \Z_4 \oplus \Z^2$.

Let $\CC$ denote the oriented cobordism group of fold maps of oriented
$4$-manifolds into $\R^3$
where two fold maps are cobordant if and only if they 
are cobordant by a $3$-simple fold map $F \co W^5 \to \R^3 \x [0,1]$
such that the
structure group of its $\si_{\iii^8}$ fiber-singularity bundle
can be reduced to the group $\Z_3$. 
%Let $\psi' \co \CC ob_{s(2)+\iii^4+\iii^6}^{O}(4,-1)
%\to \CC$ denote the natural map.

By \cite{Kal3}, the cobordism group $\CC$ is
isomorphic to $\Z_2 \oplus \Z_4 \oplus \Z^3$.
Moreover, we have a natural surjective homomorphism $\psi'' \co  \CC \to \CC ob_{s(3)}^{O}(4,-1)$ and
an injective homomorphism $\ga \co \CC \to \Z_2 \oplus \Z \oplus \Z_8 \oplus \Z_4 \oplus \Z^3$ which
goes through the group $\CC ob_{s(3)}^{O}(4,-1)$, i.e., 
let $\ga$ be the homomorphism  
$$((\de^{imm} \oplus \de^{tw} \oplus \iota^{imm} \oplus \iota^{tw}) 
\circ \psi' \circ \psi'') \oplus \pi_{\iii^5} \oplus
\pi_{\iii^7} \oplus \pi_{\iii^8},$$ where 
$\psi' \co \CC ob_{s(3)}^{O}(4,-1) \to \CC ob_{f}^{O}(4,-1)$
is the natural homomorphism and $\pi_{\iii^i} \co \CC \to \Z$
is the algebraic number of the fiber-singularity $\si_{\iii^i}$
of a class in $\CC$ ($i=5,7,8$).
Hence $\psi''$ is an isomorphism and the 
$3$-simple fold cobordim group $\CC ob_{s(3)}^{O}(4,-1)$ is isomorphic to 
$\Z_2 \oplus \Z_4 \oplus \Z^3$.

\subsection{Computing the fold cobordism group $\CC ob_{f}^{O}(4,-1)$}  

By the Pontryagin-Thom-Sz\H{u}cs type construction, the fold cobordism group
$$\CC ob_{f}^{O}(4,-1)$$ is isomorphic to
the group $\CC ob_{s(3)}^{O}(4,-1) / G$, where
$G$ is the group generated by the boundaries of
the punctured fiber-singularities $\si_{\SS}$ with $\ka(\SS)=4$ classified in \cite{SaYa}. 

Since $(\de \oplus \iota) \circ \psi'$ restricted
to the direct summand $\Z_2 \oplus \Z_4$ is injective, the 
factorization by the group $G$ can have effect only to the direct 
summand $\Z^3$. 

The boundary of $\si_{\iv^{10}}$
gives a relation between the generator of the second $\Z$
component of this direct summand $\Z^3$ (the $\Z$
component corresponding to the fiber-singularity $\si_{\iii^7}$)
and an element in the direct summand
$\Z_2 \oplus \Z_4$. Hence the factorization by the boundary
of $\si_{\iv^{10}}$ reduces the group $\CC ob_{s(3)}^{O}(4,-1) =
\Z_2 \oplus \Z_4 \oplus \Z^3$ to
$\Z_2 \oplus \Z_4 \oplus \Z^2$. Similarly, the boundary of 
 $\si_{\iv^{11}}$ reduces this last group to 
$\Z_2 \oplus \Z_4 \oplus \Z$.

In order to show that the factorization by the boundaries of the other
 fiber-singularities have no further effect, it is enough to give an 
injective homomorphism \[\CC ob_{s(3)}^{O}(4,-1) / \{ [\del \si_{\iv^{10}}], 
[\del \si_{\iv^{11}}] \} \to \Z_2 \oplus \Z \oplus \Z_8 \oplus \Z_4,\] which goes through
the fold cobordism group $\CC ob_{f}^{O}(4,-1)$.

By \cite{SaYa} the algebraic number of the fiber-singularity
$\si_{\iii^8}$ is equal to the signature of the source $4$-manifold. Since we
have already injective invariants
on the direct summand $\Z_2 \oplus \Z_4$, we obtain that
the fold cobordism group $\CC ob_{f}^{O}(4,-1)$ is isomorphic to
$\Z_2 \oplus \Z_4 \oplus \Z$.

\subsection{Generators of the cobordism group $\CC ob_{f}^{O}(4,-1)$}

By the previous proof, a system of generators consists of the cobordism classes $[f_1]$, $[f_4]$, and the cobordism class
of the fold map $f \co \CP^2 \# 2\overline{\CP^2} \to \R^3$ constructed in \cite{Sa} with one singular fiber of type $\iii^8$ and Boy's surface as the immersion of the indefinite fold singular set.
Let $\bb_1$, $\bb_2$ and $\bb_3$ denote the classes $[f_1]$, $[f_4]$ and $[\tilde f]$, respectively, where $\tilde f$
denotes the fold map obtained from $f \co \CP^2 \# 2\overline{\CP^2} \to \R^3$ by reversing the orientation of the source manifold $\CP^2 \# 2\overline{\CP^2}$.

The homomorphism $\de^{imm} \oplus \de^{tw} \oplus \iota^{imm} \oplus \iota^{tw}$ takes the values 
\[
\begin{pmatrix}
1 & 0 & 0\\
0 & 0 & 3\\ 
4 & 2 & 1\\
0 & x & y\\
\end{pmatrix}
\] 
on  the generators $\bb_1, \bb_2$ and $\bb_3$,
where the column vectors correspond to the generators $\bb_1, \bb_2, \bb_3$, and $x$ is a generator of $\Z_4$.
Hence, the homomorphism $$\de^{imm} \oplus \frac{c \circ \de^{imm} + \iota^{imm} - 
(\de^{tw}/3 \bmod 8)}{2} \oplus (\de^{tw}/3)$$ is an isomorphism between the groups $\CC ob_{f}^{O}(4,-1)$
and $\Z_2 \oplus \Z_4 \oplus \Z$, where $c \co \Z_2 \to \Z_8$ is the natural inclusion homomorphism.

Moreover, we can choose another homomorphism depending on $y$ as well, which detects the $\Z_4$ summand of 
the group $\CC ob_{f}^{O}(4,-1)$, namely the homomorphism
$\iota^{tw} - (k(\de^{tw}/3) \bmod 4)$, where $kx \equiv y \bmod 4$. Hence, for a fold map $g \co M^4 \to \R^3$ of a closed oriented $4$-manifold, we have 
$$
\frac{[f|_{S_f}] - (\si(M^4) \bmod 8)}{2} \equiv \iota^{tw}([f]) - k\si(M^4) \bmod 4.
$$
\end{proof}

\end{document}